\documentclass[reqno]{amsart}
\usepackage[english]{babel}
\usepackage{amscd,amssymb,amsmath,amsfonts,latexsym,amsthm}
\usepackage{inputenc}
\usepackage{graphicx}
\usepackage{hyperref,amsthm,amsmath,amssymb,color,multirow,graphicx}
\usepackage[fontsize=11pt]{scrextend}
\usepackage{cancel}

\newtheorem{thm}{Theorem}[section]
\newtheorem{cor}[thm]{Corollary}
\newtheorem{lem}[thm]{Lemma}
\newtheorem{prop}[thm]{Proposition}
\theoremstyle{definition}
\newtheorem{defn}[thm]{Definition}
\newtheorem{rem}[thm]{Remark}

\numberwithin{equation}{section}

\newcommand{\be}{\begin{equation}}

\newcommand{\ee}{\end{equation}}

\begin{document}

\sloppy

%\begin{flushright} 	\begin{tabular}{ll}		\textsf{Uzbek Mathematical Journal} &  \\		\textsf{20??,\ Volume ??, Issue ?, pp.\pageref{firstpage}-\pageref{lastpage}}\\		{DOI: 10.29229/uzmj.20??-?-??}\\ \end{tabular} \end{flushright}

\begin{center}
\textbf{\large Anti-Rota-Baxter operators on Witt and Virasoro algebras }\\

\textbf{ Azizov M. E.}
\end{center}

\textbf{Abstract.} In this work, we obtain the description of all homogeneous anti-Rota-Baxter operators on Witt and Virasoro algebras. Moreover, we describe anti-Rota-Baxter operators on three-dimensional simple Lie algebra $sl_2.$



\makeatletter
\renewcommand{\@evenhead}{\vbox{\thepage \hfil {\it Azizov M.}  \hrule }}
\renewcommand{\@oddhead}{\vbox{\hfill
{\it Anti-Rota-Baxter operators on Witt and Virasoro algebras }\hfill
\thepage \hrule}} \makeatother

\label{firstpage}

\section{Introduction}

The concept of Rota-Baxter operators is used in combinatorics and probability theory. Later, it was used in algebra. In 1960, Glen Baxter introduced a Baxter equation, which satisfies a certain property related to probability distributions \cite{Baxter}, using Rota-Baxter operators,
$$2(a(aT))T=(a^2b)T+(aT)^2.$$

In 1983, M.A. Semenov-Tian-Shansky studied the case of Rota-Baxter equations with weight 0 from the perspective of Lie algebras, interpreting them as Yang-Baxter equations \cite{S}. The concept of Rota-Baxter operators goes back to Richard Baxters and Gian-Carlo Rota and found several fields of mathematics \cite{Rota-Kung}.
Rota–Baxter operator of weight $0$ on a Lie algebra is exactly the operator form of the classical Yang–Baxter equation (CYBE), which was regarded as a ``classical limit'' of the quantum Yang–Baxter
equation \cite{Belavin}, whereas the latter is also an important topic in many fields such as symplectic geometry, integrable systems,
quantum groups and quantum field theory \cite{GS, Guo1, Guo2}.

By the 21st century, Rota-Baxter operators experienced a period of advancement. This can be seen in the algebraic renormalization method in quantum field theory, dendriform algebras, the associative analog of the classical Yang-Baxter equation, and it works on systems with mixed commutative and non-commutative products.

A linear operator $R: A \rightarrow A$ is called a Rota-Baxter operator on $A$ of weight $\lambda$ if the
following identity
\begin{equation}
   [R(x),R(y)] = R([R(x),y]+[x,R(y)]+ \lambda xy),
 \end{equation}
holds for any $x, y\in A.$

It is known that if $R$ is a Rota-Baxter operator weight $\lambda$ with $\lambda \neq 0,$ then the
operator $\lambda^{-1}R$ is a Rota-Baxter operator of weight 1. Moreover, if $d$ is an invertible derivation of an algebra $A$, then $d^{-1}$ forms a Rota-Baxter operator of weight zero \cite{Guo2}.

Initially, Filippov and Hopkins introduced the concept of anti-derivations in their works \cite{fil95, hop}. Later, Filippov generalized the notions of derivations and anti-derivations introducing the notion of $\delta$-derivation \cite{fil98}.
Later,  the concept of $\delta$-derivations has been further studied, giving rise to significant results \cite{fil99, zak22, pasha10}.
In particular, it was proved that prime Lie algebras, do not have nonzero $\delta$-derivations (provided $\delta \neq 1, -1, 0, \frac 1 2$), and alternative, non-Lie Malcev algebras with certain restrictions of $\mathbb{F}$ have no non-trivial $\delta$-derivation. Zusmanovich in \cite{pasha10} proved that a prime Lie superalgebra has no non-trivial $\delta$-(super)derivations for $\delta\neq  1, -1, 0, \frac 1 2$.

Based on the principal relation between derivations and Rota-Baxter operators it can be introduced the notion of $\delta$-Rota-Baxter operator weight $0$ with the condition $[R(x),R(y)] = \delta R([R(x),y]+[x,R(y)]).$
Similarly, to the derivation case, any invertible $\delta$-derivation of the algebra $A$ gives $\delta$-Rota-Baxter operator weight $0$.
On the other hand, the structure for constructing anti-$\mathcal{O}$-operators was developed by Kupershmidt \cite{Ku}, and according to it, the following definition of the anti-Rota-Baxter operator arises.

\begin{defn}  An {\bf anti-Rota-Baxter operator} on an algebra $\mathfrak{g}$
over the field $\mathbb F$ is
a linear map $R: \mathfrak{g}\rightarrow \mathfrak{g}$ satisfying
 \begin{equation}\label{RB-lambda}
   [R(x),R(y)] = -R([R(x),y]+[x,R(y)]), \qquad\forall x, y\in \mathfrak{g}.
 \end{equation}
 An anti-Rota-Baxter operator $R$ is called a {\bf strong} if R satisfies
 \begin{equation}\label{Rb-strong}
 [[R(x), R(y)], z] + [[R(y), R(z)], x] + [[R(z), R(x)], y] = 0, \quad \forall x, y, z \in \mathfrak{g}.
 \end{equation}
\end{defn}

In \cite{bas} the homogeneous Rota-Baxter operators on the Witt and Virasoro algebras were classified. In this paper, we obtain anti-Rota-Baxter operators on the Witt and Virasoro algebras by using this method. Moreover, we classify anti-Rota-Baxter operators on three-dimensional simple Lie algebra $sl_2.$ It should be noted that Rota-Baxter operators on $sl_2$ were obtained by J. Pei, C. Bai, and L. Guo in \cite{PBG}.

  \section{Homogeneous Anti-Rota-Baxter operators on the Witt and Virasoro algebra}

We also assume that $\mathbb F=\mathbb C$, is the complex field since both the Witt and Virasoro algebras are defined over $\mathbb C$.

\begin{defn}
The {\it Witt algebra} $W$ is a Lie algebra with the basis $\{L_n \ | \  n\in \mathbb{Z}\}$ and multiptications:
  \begin{equation}\label{Witt}
    [L_m,L_n]=(m-n)L_{m+n}, \quad \forall m,n \in \mathbb{Z}.
  \end{equation}
\end{defn}

There is a natural $\mathbb{Z}$-grading on the Witt algebra $W$, namely
\[
W=\bigoplus_{n\in\mathbb{Z}}W_n,
\]
where $W_n=\operatorname{span}\{L_n\}$ for any $n\in \mathbb{Z}$.

\begin{defn}\label{def:homoWitt}
Let $k$ be an integer. A {\it homogeneous operator $R$ with degree $k$} on the Witt algebra $W$ is a linear operator on $W$ satisfying
$$
R(W_m)\subset W_{m+k}, \quad \forall m\in\mathbb{Z}.
$$
\end{defn}

Therefore, a {\it homogeneous anti-Rota-Baxter operator $R_k$ with degree $k$} on the Witt algebra $W$ is an anti-Rota-Baxter operator on $W$
of the following form
\begin{equation}\label{eq:HRB}
R_k(L_m)=f(m+k)L_{m+k}, \quad \forall m\in \mathbb{Z},
\end{equation}
where $f$ is a $\mathbb C$-valued function defined on $\mathbb{Z}$.

\subsection{Homogeneous anti-Rota-Baxter operators on the Witt algebra}
\mbox{}
Let $R_k$ be a homogeneous anti-Rota-Baxter operator   with degree $k$ on the Witt algebra
$W$ satisfying equation (\ref{eq:HRB}). Then by equations ~(\ref{RB-lambda}) and (\ref{Witt}), we obtain that the function $f$ satisfies the following equation:
\begin{equation}\label{weight 0}
  f(m)f(n)(n-m)=f(m+n)(f(m)(m-n+k)+f(n)(m-n-k)),\quad \forall m,n\in \mathbb{Z}.
\end{equation}

Consider the following cases. 

\textbf{Case 1.} Let $k=0.$ Then the equation (\ref{weight 0}) becomes
\begin{equation}\label{k=0}
(m-n)\Big(f(m)f(n)+f(m+n)\big(f(m)+f(n)\big)\Big)=0,\;\;\forall m,n\in \mathbb{Z}.
\end{equation}

Putting $n=0$ in the equation, we have
\begin{equation}
mf(m)(f(m)+2f(0))=0.
\end{equation}

 If $f(0)=0,$ then we derive $f(m)=0$ for any $m\in \mathbb{Z}$.

 Now we assume $f(0)\neq 0$ and define following sets:$$I=\lbrace m\in \mathbb{Z}\setminus \lbrace 0 \rbrace : f(m)=0\rbrace, \qquad J=\lbrace m\in \mathbb{Z}\setminus \lbrace 0 \rbrace : f(m)=-2f(0)\rbrace .$$ 
 
 Since $f(0) \neq 0,$ we easily obtain that $I \bigcap J = \varnothing $ and $I \bigcup J = \mathbb {Z} \setminus \lbrace 0 \rbrace .$

\begin{lem}\label{-n}
    Let $n$ be a nonzero integer, then $n \in I $ if and only if $-n \in I.$
\end{lem}
\begin{proof}
Taking $m=-n$ in (\ref{k=0}), we get $$f(-n)f(n)+f(0)\big(f(-n)+f(n)\big)=0.$$ 

From this equation we obtain that $f(n)=0$ if and only if $f(-n)=0.$
\end{proof}

% \begin{lem}\label{m+n}
%     Let $n\in J$ then, for all $m \neq \pm n:$ $m+n \in I$ iff $m\in I.$
% \end{lem}
% \begin{proof}
%     If $m+n \in I, m \neq n$ then by ~(\ref{k=0}): $(n-m)f(n)f(m)=0.$ From $f(n)\neq 0$ immediately, implies $f(m)=0.$

%     If $m\in I, m \neq \pm n$ then by ~(\ref{k=0}): $(n-m)f(m+n)f(n)=0,$ with same reason implies $f(m+n)=0.$
% \end{proof}

\begin{prop}\label{prop1}
The anti-Rota-Baxter operator $R_0$ with degree $0$
is given by
\begin{equation*}
R_0(L_n)= \begin{cases}
0,   \quad  \quad  n \neq 0,\\
\alpha L_0,\quad  n=0,
\end{cases}
\end{equation*}
where $\alpha\in \mathbb{C}$.
\end{prop}
\begin{proof}
% Let $1 \in J$ then by Lemma \ref{-n}: $-1\in J$ i.e. $f(1)=f(-1)=-2f(0).$ We consider equation \ref{k=0} with case $m=2, n=-1$: \[ f(2-1)(f(2)+f(-1))+f(2)f(-1)=0 \Rightarrow f(2)+f(-1)+f(2)=0 \Rightarrow f(2)=\frac{-1}{2}f(-1)=f(0),\] contradiction. Then $1$ must belongs to $I.$

% Now we take the case  m=1:
% \[f(n+1)(f(n)+f(1))+f(n)f(1)=0 \Rightarrow f(n+1)f(n)=0\] with the help of this argument we get next condition:\[f(n+1+n)(f(n+1)+f(n))+f(n+1)f(n)=0 \Rightarrow f(2n+1)(f(n+1)+f(n))=0.\] Now we have two possible solutions: $f(2n+1)=0$ or $f(n+1)=-f(n).$ For some $n.$

Let $f(p)=-2f(0)\neq 0,$ for $p \neq 0,$ then putting $m=2p, n=-p$ in (\ref{k=0}), we have  \[ f(2p-p)(f(2p)+f(-p))+f(2p)f(-p)=0.\] 

Then by Lemma \ref{-n}, we get that $f(-p)=f(p)=-2f(0) \neq 0,$
which implies $2f(2p)+f(-p)=0$. Then, we get $f(2p)=-\frac{1}{2}f(-p)=f(0).$ But $f(2p)=0$ or $f(2p) = -2f(0).$ From this, we get $f(0)=0,$ which is a contradiction. Hence, $f(p)=0$ for all $p\neq 0.$
\end{proof}

\textbf{Case 2.} Let $k\ne 0.$ In the case of $f(0)=0,$ we have the following Proposition.
\begin{prop}\label{prop2}
Let $k\neq 0$ and $f$ satisfies the equation (\ref{weight 0}). If $f(0)=0$, then
\begin{equation*}
f(m)=\begin{cases}
0, & m \neq -k,\\
\alpha,& m=-k,
\end{cases}
\end{equation*}
where $\alpha\in \mathbb{C}$.
\end{prop}

\begin{proof}
If $f(0)=0$, then taking $n=0$ in the equation (\ref{weight 0}),
we have
\begin{equation*}
  (m+k)(f(m))^2=0,\;\;\forall m\in \mathbb{Z}.
\end{equation*}

Thus, $f(m)=0$ for $m \neq -k$ and $f(-k)$ can be any complex value. 
\end{proof}

Now consider the case of $f(0) \neq 0.$ Then putting $m=0, n= \frac{k}{2},$ in (\ref{weight 0}), we get $f(\frac{k}{2})=0.$ Now setting $n=0$ in (\ref{weight 0}), we have
$$f(m)\Big(f(0)(2m-k)+(m+k)f(m)\Big)=0.$$

%From this, it is easy to see that $f(\frac {k} {2}) =0.$
% Now consider $f(0)\neq 0.$ Then from (\ref{weight 0}) we get $f(\frac{k}{2})=0$. Now supplying $n=0$ in (\ref{weight 0}), we get following equation:
% $$f(m)(f(0)(2m-k)+(m+k)f(m))=0.$$

We set $\mathcal{I}$ and $\mathcal{J}:$ \[\mathcal{I}=\lbrace m\in \mathbb{Z} : f(m)=0 \rbrace, \qquad \mathcal{J}=\lbrace m\in \mathbb{Z}\setminus \{ \tfrac{k}{2} \}: f(0)(2m-k)+(m+k)f(m)=0 \rbrace .\] For these sets we can write $\mathcal{I} \bigcup \mathcal{J}=\mathbb{Z}, \ \ \mathcal{I} \bigcap \mathcal{J}=\varnothing.$

%%%%%%%%%%%%%%%%%%%%%%%%%%%%%%%%%%%%%%%%%%%%
%%%%%%%%%%%%%%%%%%%%%%%%%%%%%%%%%%%%%%%%%%%%
%%%%%%%%%%%%%%%%%%%%%%%%%%%%%%%%%%%%%%%%%%%%

\begin{lem}\label{lemma2.5}
Let $f$ be a $\mathbb C$-valued function defined on $\mathbb{Z}$ satisfying equation~(\ref{weight 0}). Suppose that $f(0)\neq 0$ and $k\ne 0$. If $n\in \mathcal{J}$ and
$m\neq n, n+k$, then $m\in \mathcal{I}$ if and only if $m+n\in \mathcal{I}$.
\end{lem}

\begin{proof}
If $m\in \mathcal{I}$, $m\neq n+k$ and $n\in \mathcal{J}$, then by equation~(\ref{weight 0}), we have
\begin{align*}
(m-n-k)f(n)f(n+m)=0.
\end{align*}
Since $n\in \mathcal{J}$ and $m\ne n+k,$ we have $f(n+m)=0$.

Conversely, if $m+n\in \mathcal{I}$, then
by equation ~(\ref{weight 0}), we have
$$(m-n)f(m)=0.$$
Since $n\in \mathcal{J}$ and $m\neq n,$ we have $f(m)=0$.
Hence $m\in \mathcal{I}$.
\end{proof}

For an integer $m\in \mathbb{Z}$, we set
\[
\mathcal{J}_m=\{n\in \mathcal{J} \ | \ mn\in \mathcal{J}\},\;\;
\mathcal{I}_m=\{n\in \mathcal{J} \ | \ mn+k\in \mathcal{I}\}.
\]

\begin{prop}\label{prop2.6} For the set $\mathcal{J}_m$, we have the following:
\begin{enumerate}
  \item $\mathcal{J}_0=\mathcal{J}_1=\mathcal{J}$.
  \item $(\mathcal{J}\setminus \{-\tfrac{k}{2m}\})\cap \mathcal{J}_m\subset \mathcal{J}_{-m}$ for every $m\neq 0$.
  In particular, $\mathcal{J}\setminus \{-\tfrac{k}{2}\}\subset \mathcal{J}_{-1}$.
  \item $(\mathcal{J}\setminus \{-\tfrac{k}{2}, \tfrac{k}{m+1}\})\cap \mathcal{J}_{m-1}\subset \mathcal{J}_m$,
  $(\mathcal{J}\setminus \{-\tfrac{k}{m+1}\})\cap \mathcal{J}_{1-m}\subset \mathcal{J}_{-m}$ for $m\geq 2$.
  \item $(\mathcal{J}\setminus \{\tfrac{k}{2m-1}\})\cap \mathcal{J}_{1-m}\subset \mathcal{J}_{m}$,
  $(\mathcal{J}\setminus \{-\tfrac{k}{2}, -\tfrac{k}{2m-1}\})\cap \mathcal{J}_{m-1}\subset \mathcal{J}_{-m}$ for $m\geq 2$.
\end{enumerate}
\end{prop}

\begin{proof}
Similarly to the proof [\cite{bas}. Proposition 2.7]
% {\begin{enumerate}
% \item follows immediately by definition.
% \item In fact, it is straightforward to check that $0\in (\mathcal{J}\setminus \{-\tfrac{k}{2m}\})\cap \mathcal{J}_m$
% and $0\in \mathcal{J}_{-m}$ for $m\neq 0$.
% Let $n\neq 0$ and $n\in (\mathcal{J}\setminus \{-\tfrac{k}{2m}\})\cap \mathcal{J}_m$.
% To prove (2), we only need to show
% that $-nm\in \mathcal{J}$. Otherwise,  $-nm\in \mathcal{I}$. Then by Lemma~\ref{lemma2.5},  we have  $nm-nm=0\in \mathcal{I}$,
% which is a contradiction with the assumption that $f(0)\neq 0$.
% \item Let $\forall n \in \mathcal{J}\setminus \lbrace \frac{-k}{2}, \frac{k}{m+1} \rbrace \bigcap \mathcal{J}_{m-1}.$ Assume that $nm \in \mathcal{I}.$ Then by Lemma \ref{lemma2.5},\\
%  $(nm-n)+n=nm \in \mathcal{I},$ implies $nm-n \in \mathcal{I}$   contradiction. Thus $nm \in \mathcal{J}\Rightarrow n \in \mathcal{J}_{m} $
% \item Let $\forall n \in (\mathcal{J}\setminus \{\tfrac{k}{2m-1}\})\cap \mathcal{J}_{1-m} \Rightarrow n-nm \in \mathcal{J}, n\in \mathcal{J} .$ Let $nm \in \mathcal{I}.$ Then by Lemma \ref{lemma2.5}, $n-nm+nm=n \in \mathcal{I}$ which is contradiction. So, $nm \in J \Rightarrow n \in J_{m}.$
% \end{enumerate}
% }
\end{proof}

\begin{cor}\label{cor2.7} $
\mathcal{J}\setminus \{-\tfrac{k}{2}\}\subset \bigcap\limits_{m\in \mathbb{Z}} \mathcal{J}_m.$
\end{cor}

\begin{proof}
Similarly to the proof [\cite{bas}. Corollary 2.8]
\end{proof}

\begin{prop}\label{prop2.8} For the set $\mathcal{I}_m$, we have the following:
\begin{enumerate}
  \item $\mathcal{I}_0=\mathcal{J}$.
  \item $\mathcal{J}\setminus \{-\tfrac{k}{2},\tfrac{k}{m}\}\subset \mathcal{I}_m$ for any $m\neq 0$.
  \item $\mathcal{J}\setminus \{-\tfrac{k}{2},-\tfrac{k}{2m}\}\subset \mathcal{I}_m$ for any $m\neq 0$.
\end{enumerate}
\end{prop}
\begin{proof}
Similarly to the proof [\cite{bas}. Proposition 2.9]
% (1) follows from the fact that $k\in \mathcal{I}$.  We only give a detailed proof of (3) and the proof of (2) is similar.

% Let $m$ be a fixed non-zero integer.  Since $0\in \mathcal{J}$ and $k\in \mathcal{I}$, we show that $0\in \mathcal{I}_m$.
% Let $n_0$ be an arbitrary nonzero integer in $\mathcal{J}\setminus \{\tfrac{-k}{2},\tfrac{-k}{2m}\}$.
% Then we have $k+mn_0\neq -mn_0$ and
% $k+mn_0\neq -mn_0+k$. By Corollary~\ref{cor2.7}, we have $-mn_0\in \mathcal{J}$. Hence by Lemma~\ref{lemma2.5} and since $mn_0+k-mn_0=k\in \mathcal{I}$, we have
% $mn_0+k\in \mathcal{I}$.
\end{proof}

By Proposition~\ref{prop2.8}, we get the following result.
\begin{cor}\label{cor2.9} $
\mathcal{J}\setminus \{-\tfrac{k}{2}\}\subset \bigcap\limits_{m\in \mathbb{Z}} \mathcal{I}_m.$
\end{cor}

\begin{prop}\label{prop2.10} Let $n\in \mathcal{J}\setminus \{-\tfrac{k}{2}\}$ and $n\ne 0$. Then we have $n\nmid k$, and
for any $m\in \mathbb{Z}$,
\begin{enumerate}
  \item if $m\in \mathcal{I}$, then $m+n\mathbb{Z}\subset \mathcal{I}$;
  \item if $m\in \mathcal{J}$, then $m+n\mathbb{Z}\subset \mathcal{J}$.
\end{enumerate}
\end{prop}
% \begin{proof}
% If $m$ is neither in $n\mathbb{Z}$ nor in $k+n\mathbb{Z}$, the conclusion holds due to Lemma~\ref{lemma2.5}.
% On the other hand, by Corollary \ref{cor2.7}, we show that $n\in \bigcap\limits_{l\in \mathbb{Z}} \mathcal{J}_l$.
% Hence $n\mathbb{Z}\subset \mathcal{J}$.
% Furthermore, by Corollary \ref{cor2.9} and the fact that $k\in \mathcal{I}$, we have $n\in \bigcap\limits_{l\in \mathbb{Z}} \mathcal{I}_l$.
% Thus $k+n\mathbb{Z}\subset \mathcal{I}$. Therefore for any $m\in n\mathbb{Z}$ or $m\in k+n\mathbb{Z}$, if $m\in \mathcal{I}$, then $m\in k+n\mathbb{Z}$ and hence
% $m+n\mathbb{Z}\in \mathcal{I}$, and if $m\in \mathcal{J}$, then $m\in n\mathbb{Z}$ and hence $m+n\mathbb{Z}\in \mathcal{J}$.
% Moreover $n\nmid k$. Otherwise we have $n\mathbb{Z}=k+n\mathbb{Z}\subset \mathcal{I}\cap \mathcal{J}$, which is a contradiction.
% \end{proof}
% Own proof:
\begin{proof}
  Similarly to the proof [\cite{bas}. Proposition 2.11]
    % Let $m \notin n\mathbb{Z}$ and $m \notin k+n\mathbb{Z}.$ Now assume $m\in \mathcal{I}.$ From Corollary \ref{cor2.7} $n\in \bigcap_{m\in \mathbb{Z}} \mathcal{J}_m$ it means that $n\mathbb{Z} \subset \mathcal{J}.$ Then by using Lemma \ref{lemma2.5} we get $m+n\mathbb{Z} \subset \mathcal{I}.$

    % Now, $m\in \mathcal{J}.$ Lets assume that $m+n\mathbb{Z} \subset I,$ then by Lemma \ref{lemma2.5}   $n \mathbb{Z} \subset I,$ contradiction. This gives us $m+n\mathbb{Z} \subset J.$

    % Let $m \in n\mathbb{Z}$ or $m \in k+n\mathbb{Z}.$ If $m \in n\mathbb{Z} \Rightarrow$ since Corollary \ref{cor2.9} $m \in \mathcal{J} \Rightarrow m+n\mathbb{Z}\subset n\mathbb{Z} \subset \mathcal{J}.$  If $m \in k+n\mathbb{Z}\Rightarrow$ since Corollary \ref{cor2.7} $m\in \mathcal{I} \Rightarrow m+n\mathbb{Z}\subset k+n\mathbb{Z} \subset \mathcal{I}.$
\end{proof}

For any $m,n\in\mathbb{Z}$ (not both zero), let ${\rm gcd}(m,n)$ denote the greatest common divisor of $m$ and $n$.

\begin{cor}\label{cor2.11} If $n_1\in \mathcal{J}$, $n_2\in \mathcal{J}\setminus \{0,-\tfrac{k}{2}\}$, then ${\rm gcd}(n_1,n_2)\mathbb{Z}\subset \mathcal{J}$.
\end{cor}
\begin{proof}
Similarly to the proof [\cite{bas}. Corollary 2.12]
% If $n_1\neq \tfrac{-k}{2}$, then by Proposition~\ref{prop2.10}, we show that for every $m_1$, $m_2\in\mathbb{Z}$,
% $n_1m_1+n_2m_2\in \mathcal{J}$. Furthermore, we have ${\rm gcd}(n_1,n_2)\mathbb{Z}=n_1\mathbb{Z}+n_2\mathbb{Z}$. Thus
% ${\rm gcd}(n_1,n_2)\mathbb{Z}\subset \mathcal{J}$.

% If $n_1=\tfrac{-k}{2}\in \mathcal{J}$, then by Proposition~\ref{prop2.10}, we show that $n_1+n_2\in \mathcal{J}$.
% On the other hand, we have $n_2$, $n_1+n_2\in \mathcal{J}\setminus \{\tfrac{-k}{2}\}$. Hence gcd$(n_1+n_2, n_2)\mathbb{Z}\subset \mathcal{J}$.
% Since gcd$(n_1+n_2, n_2)\mathbb{Z}={\rm gcd} (n_1,n_2)\mathbb{Z}$, we have ${\rm gcd} (n_1,n_2)\mathbb{Z}\subset \mathcal{J}$.
\end{proof}

\begin{prop}\label{prop3}
Let $f$ be a $\mathbb C$-valued function defined on $\mathbb{Z}$ satisfying equation~(\ref{weight 0}). Suppose that $f(0)\neq 0$ and $k\ne 0$.
If $\tfrac{-k}{2}\in \mathcal{J}$, then $\mathcal{J}=\{0,-\tfrac{k}{2}\}$, and in this case,
\begin{equation}f(m)=\delta_{m,0}f(0)+4
\delta_{m,\frac{-k}{2}}f(0),\;\;\forall m\in \mathbb{Z}.\end{equation}
\end{prop}
\begin{proof}
It is obvious that $\{0, -\tfrac{k}{2}\}\subset \mathcal{J}$. Conversely, if there exists an $n_0\in \mathcal{J}$
such that $n_0\neq 0$, $-\tfrac{k}{2}$,  then by Corollary \ref{cor2.11}, we have
${\rm gcd} (n_0,-\tfrac{k}{2})\mathbb{Z}\subset \mathcal{J}$. Since $\mathcal{J}\neq \mathbb{Z}$, we have
${\rm gcd} (n_0,-\tfrac{k}{2})\neq 1$. Set $d={\rm gcd}(n_0,-\tfrac{k}{2})$. Then $d|\tfrac{-k}{2}$. Hence
$d|k$. By Proposition~\ref{prop2.10}, we show that $d=-\tfrac{k}{2}$. Thus $n_0=\tfrac{k}{2}m_0$
for some $m_0\neq 0$, $-1$. However, by Lemma~\ref{lemma2.5} and induction on $m$ (note that $\pm k\in \mathcal{I}$), one can show that
$\tfrac{k}{2}m\in \mathcal{I}$ for any $m\neq0$, $-1$. It is a contradiction.
Hence $\mathcal{J}=\{0,-\tfrac{k}{2}\}$.
\end{proof}

\begin{prop}\label{prop4}
Let $f$ be a $\mathbb C$-valued function on $\mathbb{Z}$ satisfying equation~(\ref{weight 0}). Suppose that $f(0)\neq 0$ and $k\ne 0$.
If $\tfrac{-k}{2}\notin \mathcal{J}$, and $\{0\}\subsetneqq \mathcal{J}$, then
there exists a non-zero integer $l\in \mathcal{J}$, $l\nmid k$ such that $|l|$ is minimal.
In this case, we have
\begin{align*}
\mathcal{J}=l\mathbb{Z},
\end{align*}
and thus
 \begin{equation}
        f(m)=
        \tfrac{k-2m}{m+k}\delta_{m,l\mathbb{Z}}f(0),
    \end{equation}
where
\[
\delta_{m,l\mathbb{Z}}:=
\sum_{n\in\mathbb{Z}}\delta_{m,ln} =
\begin{cases}
1 & m\in l\mathbb{Z};\\
0 & m\notin l\mathbb{Z}.
\end{cases}
\]
\end{prop}
\begin{proof}
Since $ \{0\}\subsetneqq \mathcal{J}$, there exists an integer $l\in \mathcal{J}$ such that $l\neq 0$, $|l|$ is minimal and $l\nmid k$.
By Proposition~\ref{prop2.10} and the minimality of $|l|$, we have $m\in \mathcal{I}$ for any $m\notin l\mathbb{Z}$.
On the other hand, since $0\in \mathcal{J}$ and by Proposition~\ref{prop2.10}, we have $l\mathbb{Z}\subset \mathcal{J}$. Hence $\mathcal{J}=l\mathbb{Z}$ and thus the conclusion holds.
\end{proof}

In summary, we obtain the following Theorem.

\begin{thm}\label{Thm2.14}
A homogeneous anti-Rota-Baxter operator on the Witt algebra $W$ has one of the following forms:
\begin{enumerate}
	\item[(I)] $R_k^{\alpha}(L_m) = \alpha\delta_{m+2k,0}L_{m+k},\;\;\forall m\in\mathbb{Z}$,
	where $k\in\mathbb{Z}$ and $\alpha \in \mathbb C$.
	
	\item[(II)] $R_{2k}^{'\beta}(L_m) = (\beta\delta_{m+2k,0}+4\beta\delta_{m+3k,0})L_{m+2k},\;\;\forall m\in\mathbb{Z}$,
	where $k\in\mathbb{Z}^{\ast}$ and $\beta \in \mathbb{C}^{\ast}$.

	\item[(III)] $R_k^{l,\gamma}(L_m) = \tfrac{k-2m}{m+2k}\gamma\delta_{m+k,l\mathbb{Z}} L_{m+k},\;\;\forall m\in\mathbb{Z}$,
	where $k,l\in\mathbb{Z}^{\ast}$, $l\nmid k$ and $\gamma \in \mathbb{C}^{\ast}$.
\end{enumerate}
Moreover,
\begin{enumerate}
	\item $\left\{R_0^{\alpha}
	\middle|\alpha\in\mathbb{C}
	\right\}$ are all the homogeneous anti-Rota-Baxter
	operators with degree $0$ on the Witt algebra $W$.
	\item If $k\neq 0$ and is odd, then
	$\left\{R_k^{\alpha}, R_k^{l,\beta}
	\middle|\alpha\in\mathbb{C}, \beta\in\mathbb{C}^{\ast}, l\in\mathbb{Z}^{\ast}, l\nmid k
	\right\}$
	are all the homogeneous anti-Rota-Baxter
	operators   with degree $k$ on $W$.
	\item If $k\neq 0$ and is even, then
	$\left\{R_k^{\alpha}, R_k^{'\beta}, R_k^{l,\gamma}
	\middle|\alpha\in\mathbb{C}, \beta,\gamma\in\mathbb{C}^{\ast}, l\in\mathbb{Z}^{\ast}, l\nmid k
	\right\}$
	are all the homogeneous anti-Rota-Baxter
	operators   with degree $k$ on $W$.
\end{enumerate}
\end{thm}

\begin{proof}
The first part follows from Propositions\ref{prop1}, \ref{prop2}, \ref{prop3} and \ref{prop4}. The second part
can be directly verified.
\end{proof}

\subsection{Homogeneous anti-Rota-Baxter operators on the Virasoro algebra}

\begin{defn}
The {\it Virasoro algebra} $V$ is a Lie algebra with the basis
$\{L_m,C|m\in\mathbb{Z}\}$ satisfying the following relations:
\begin{equation}\label{Virasoro_1}
    [L_m,L_n]=(m-n)L_{m+n}+\tfrac{m^3-m}{12}\delta_{m+n,0}C, \quad   \forall m,n \in \mathbb{Z}.
\end{equation}
\end{defn}
%\begin{equation}\label{Virasoro_2}[L_m,C]=0,\;\;\forall m\in\mathbb{Z}.\end{equation}

The Virasoro algebra $V$ is a central extension of the Witt algebra $W$,
and has a natural $\mathbb{Z}$-grading as well:
\begin{align*}
V=\bigoplus_{n\in \mathbb{Z}} V_n,
\end{align*}
where $V_n=\operatorname{span}\{L_n\}$ for $n\in \mathbb{Z}^{\ast}$ and $V _0=\operatorname{span}\{L_0, C\}$.

% \begin{rem}\label{W and V}
% 	The Witt algebra $W$ is a graded quotient of the Virasoro algebra $V$.
% 		Any linear graded operator on $W$ can be lifted to that of $V$ by mapping $\mathbb CC$ to $0$.
% 	Conversely, any linear graded operator $F$ on $V$ can be restricted to a linear operator on $W$ by forgetting the image of $\mathbb CC$.
% 	If the kernel of $F$ contains the center, then the two can be identified.
% \end{rem}

% \begin{defn}\label{homogeneous R-B on Vir}
% 	Let $k$ be an integer. A {\it homogeneous operator $F$ with degree $k$} on the Virasoro algebra $V$ is a linear operator on $V$ satisfying
% 	\[
% 		F(V_m)\subset V_{m+k},\;\;
% 		\forall m\in\mathbb{Z}.
% 	\]
% \end{defn}

Hence {\it homogeneous anti-Rota-Baxter operator $R_k$ with degree $k$} on the Virasoro algebra $V$ is an operator on $V$ with the following form:
\begin{equation}\label{eq:RBV1}
R_k(L_m)=f(m+k)L_{m+k}+\theta\delta_{m+k,0}C,\;\;\forall m\in\mathbb{Z};
\end{equation}
\begin{equation}\label{eq:RBV2}
R_k(C)=\mu L_k+\nu\delta_{k,0}C,
\end{equation}
where $f$ is a $\mathbb C$-valued function defined on $\mathbb{Z}$ and $\theta, \mu, \nu\in \mathbb C$.

\begin{thm}\label{Thm3.3}
A homogeneous anti-Rota-Baxter operator $R_0$  with degree $0$ on the Virasoro algebra $V$ must be of the form
\begin{align*}
    R_0^{\alpha,\theta,\mu,\nu}(L_m)&=\delta_{m,0}(\alpha L_m+\theta C),\;\;\forall m\in \mathbb{Z},\\
    R_0^{\alpha,\theta,\mu,\nu}(C)&=\mu L_0+\nu C,
  \end{align*}
where $\alpha, \theta, \mu, \nu \in \mathbb C$ are arbitrary. Conversely, the above operators are all the homogeneous
 anti-Rota-Baxter operators  with degree $0$ on the Virasoro algebra $V$.
\end{thm}
\begin{proof}
    Similarly to the proof [\cite{bas}. Theorem 3.4].
\end{proof}

\begin{thm}\label{Thm3.6}
	A homogeneous anti-Rota-Baxter operator with a
	nonzero degree on the Virasoro algebra $V$ has one of the following forms:
	\begin{enumerate}
		\item[(I)]
		$
		R_k^{\theta}(L_m)=
		\theta\delta_{m+k,0}C
		,\;\;\forall m\in \mathbb{Z}$ and
		$R_k^{\theta}(C)=0$,
		where $k\in \mathbb{Z}^{\ast}$ and $\theta \in\mathbb{C}$.
		\item[(II)]
		$
		R_k^{\alpha}(L_m)=\alpha\delta_{m+2k,0}L_{m+k},\;\;\forall m\in \mathbb{Z}$ and
		$R_k^{\alpha}(C)=0$, where $k\in \mathbb{Z}^{\ast}$ and $\alpha \in\mathbb{C}^{\ast}$
		\item[(III)]
		$
		R_{2k}^{'\beta+\vartheta}(L_m)=
		(\beta\delta_{m+2k,0}+4\beta\delta_{m+3k,0})L_{m+2k}+\vartheta\delta_{m+2k,0}C,\;\;\forall m\in\mathbb{Z}$ and
		$R_{2k}^{'\beta+\vartheta}(C)=0$, where
		$k\in \mathbb{Z}^{\ast}$, $\beta\in\mathbb{C}^{\ast}$ and $\vartheta\in\mathbb{C}$.
		\item[(IV)]
		$
		R_k^{\mu}(L_m)=
		\tfrac{k^2-1}{24}\mu\delta_{m,0} L_{m+k}
		,\;\;\forall m\in\mathbb{Z}$, and $R_k^{\mu}(C)=\mu L_k$,
		where $\mu\in\mathbb{C}^{\ast}$.
	\end{enumerate}
	Moreover,
	\begin{enumerate}
		\item If $k\neq 0$ and is odd, then
		$\left\{R_k^{\theta}, R_k^{\alpha}, R_k^{\mu}
		\middle|\alpha,\mu\in\mathbb{C}^{\ast}, \theta\in\mathbb{C}
		\right\}$
		are all the homogeneous anti-Rota-Baxter
		operators  with degree $k$ on the Virasoro algebra $V$.
		\item If $k\neq 0$ and is even, then
		$\left\{R_k^{\theta}, R_k^{\alpha}, R_{k}^{'\beta+\vartheta}, R_k^{\mu}
		\middle|\alpha,\beta,\mu\in\mathbb{C}^{\ast}, \theta,\vartheta\in\mathbb{C}
		\right\}$
		are all the homogeneous anti-Rota-Baxter
		operators  with degree $k$ on the Virasoro algebra $V$.
	\end{enumerate}
\end{thm}
\begin{proof}
    Similarly to the proof [\cite{bas}. Theorem 3.9].
\end{proof}

\section{Anti-Rota-Baxter operators on the algebra  $sl(2).$}

Let's consider the following 3-dimensional Lie algebra $sl_2(\mathbb{C})=\lbrace e_1, e_2, e_3 \rbrace$ with the table of multiplications:
\begin{equation}\label{sl2}
    \left[ e_1,e_2 \right]=e_3, \qquad \left[e_1,e_3 \right]= 2e_1, \qquad \left[ e_2,e_3 \right]=-2e_2, \qquad
\end{equation}

 Let $R$ be an anti-Rota-Baxter operator on $sl_2(\mathbb{C})$ and its matrix on the basis $\lbrace e_1, e_2, e_3 \rbrace$ has the form:
$$
\left( \begin{matrix}
    a & b & c \\
    d & g & h \\
    k & l & m
\end{matrix} \right) .
$$

\begin{thm}
    The matrix of the anti-Rota-Baxter operator on the algebra $sl_2(\mathbb{C})$ has one of the following form:
$$\left( \begin{matrix}
   0 & 0 & 0  \\
   d & 0 & 0  \\
   0 & 0 & m  \\
\end{matrix} \right), \
\left( \begin{matrix}
   0 & 0 & 0  \\
   d & 0 & h  \\
   0 & 0 & 0  \\
\end{matrix} \right) (h \neq 0), \ 
\left( \begin{matrix}
   0 & b & c  \\
   0 & 0 & 0  \\
   0 & 0 & 0  \\
\end{matrix} \right) (c \neq 0), \
\left( \begin{matrix}
   0 & b & 0  \\
   0 & 0 & 0  \\
   0 & 0 & m  \\
\end{matrix} \right) (b \neq 0), 
$$
$$
\left( \begin{matrix}
   0 & 0 & 0  \\
   d & 0 & -\frac{k}{2}  \\
   k & 0 & m  \\
\end{matrix} \right) (k \neq 0), \
\left( \begin{matrix}
   0 & b & 0  \\
   0 & 0 & -\frac{k}{2}  \\
   k & 0 & m  \\
\end{matrix} \right) (b \cdot k \neq 0), \
\left( \begin{matrix}
   0 & b & c  \\
   d & 0 & \frac{bd}{4c}  \\
   -\frac{bd}{2c} & -2c & m  \\
\end{matrix} \right) (c \neq 0), \
\left( \begin{matrix}
   a & -\frac{l^2}{4a} & 0  \\
   \frac{4a^3}{l^2} & -a & 0  \\
   -\frac{4a^2}{l} & l & 0  \\
\end{matrix} \right)(a \cdot l \neq 0),
$$
$$
\left( \begin{matrix}
  a & b & c  \\
  d & a & h  \\
  -2h & -2c & \frac{b d-a^2-4 c h}{2 a} \\
\end{matrix} \right)(a \neq 0),
\left( \begin{matrix}
   a & \frac{4 c^2 g}{(a+g)^2} & c  \\
   \frac{a (a+g)^2}{4 c^2} & g & \frac{(a+g)^2}{4 c} \\
  -\frac{a (a+g)}{c} & -\frac{4 c g}{a+g} & -a-g \\
\end{matrix} \right)(c \neq 0, a \neq \pm g).
$$

Among these operators, the following are strong anti-Rota-Baxter:
$$\left( \begin{matrix}
   0 & 0 & 0  \\
   d & 0 & 0  \\
   0 & 0 & m  \\
\end{matrix} \right), \
\left( \begin{matrix}
   0 & 0 & 0  \\
   d & 0 & h  \\
   0 & 0 & 0  \\
\end{matrix} \right) (h \neq 0),
\left( \begin{matrix}
   0 & 0 & 0  \\
   d & 0 & -\frac{k}{2}  \\
   k & 0 & m  \\
\end{matrix} \right) (k \neq 0), 
\left( \begin{matrix}
   0 & b & 0  \\
   0 & 0 & -\frac{k}{2}  \\
   k & 0 & m  \\
\end{matrix} \right) (b \cdot k \neq 0).
$$
\end{thm}
\begin{proof}
   
 Let $R$ be an anti-Rota-Baxter operator on $sl_2(\mathbb{C}).$ Then we have 
    \begin{equation}\label{suit}
        [R(e_i),R(e_j)]+R([R(e_i), e_j]+[e_i, R(e_j)])=0, \quad 1 \leq i,j\leq 3.
    \end{equation}

Considering all possible cases for $i$ and $j,$ we obtain the following relations:
\begin{equation} \label{sys1}
\begin{array}{|l|l|l|}
     \hline
   \text{for } i=1,j=2& \Rightarrow & \begin{array}{l}
    4 a h +(a +g) k =0 \\
    4 c g +(a +g) l=0 \\
    -b d +a g +4 c h +a m +g m= 0 \\
   \end{array}  \\
   \hline
   \text{for } i=1,j=3& \Rightarrow &
   \begin{array}{l}
 2 a^2 -2 b d -2 c k +k l +4 a m = 0 \\
    2 a b -2 b g +2 c l +l^2 =0 \\
    2 a c -2 b h -b k +a l +2 c m +l m =0 \\
   \end{array}  \\
       \hline

   \text{for } i=2,j=3& \Rightarrow & \begin{array}{l}
    2 a d -2 d g-2 h k -k^2 =0 \\
    2 b d -2 g^2 +2 h l -k l -4 g m =0 \\
    2 c d -2 g h -g k +d l -2 h m -k m=0 \\
   \end{array}  \\
        \hline
     \end{array}
    \end{equation}

Firstly, let us observe $ 4 c g +(a +g) l=0. $ From this, we derive that if $a+g=0,$ then $g=0$ or $c=0.$ Thus, we consider following cases:

% $\bullet$ If $b=0,$ then $l=0$ or $l=2c \neq 0$

\textbf{Case 1.} Let $a=g=0.$ Then we get

\begin{equation}\label{sys2}
    \left\{\begin{array}{rl}
    4ch-bd = 0, \\
    2 b d+k (2 c-l) = 0, \\
    l(2 c+l)= 0, \\
    b (2 h+k)-m(2 c+l) = 0, \\
    k (2 h+k) = 0, \\
    2 b d+l(2 h -k) = 0,\\
    d(2 c +l)-m(2 h+k)= 0.
    \end{array} \right.
\end{equation}

 Now we obtain three possible subcases:

\qquad {\bf Case 1.1.} Let $l=0,$ $k=0.$ Then we have
\begin{equation}\label{sys2}
    \left\{\begin{array}{rl}
    4ch-bd = 0, \\
     b d= 0, \\
    bh-cm = 0, \\
    c d- h m= 0.
    \end{array} \right.
\end{equation}

If $b=0,$ then we get $ch=0,$ $cm=0,$ $cd=hm,$ which implies

 -- when $c=0:$
 $$\left( \begin{matrix}
   0 & 0 & 0  \\
   d & 0 & 0  \\
   0 & 0 & m  \\
\end{matrix} \right) , \quad 
\left( \begin{matrix}
   0 & 0 & 0  \\
   d & 0 & h  \\
   0 & 0 & 0  \\
\end{matrix} \right) (h \neq 0).$$

-- when $c\neq 0,$ then we get $m=d=0,$ $h=0$ and obtain
$$\left( \begin{matrix}
   0 & 0 & c  \\
   0 & 0 & 0  \\
   0 & 0 & 0  \\
\end{matrix} \right).$$

If $b\neq 0,$ then $d=0,$ $h=0$ and $ c m=0.$ So we get:
$$\left( \begin{matrix}
   0 & b & 0  \\
   0 & 0 & 0  \\
   0 & 0 & m  \\
\end{matrix} \right) (b \neq 0), \quad 
\left( \begin{matrix}
   0 & b & c  \\
   0 & 0 & 0  \\
   0 & 0 & 0  \\
\end{matrix} \right) (b \cdot c \neq 0).$$

\qquad {\bf Case 1.2.} Let $l=0,$ $k \neq 0.$ Then $h=-\frac{k}{2}$ and we have 
\begin{equation}\label{sys2}
    \left\{\begin{array}{rl}
    c k = 0, \\
    c m = 0,  \\
    b d = 0, \\
    c d= 0.
    \end{array} \right.
\end{equation}

From this, we have $c=0$ and obtain the following two solutions:
$$
\left( \begin{matrix}
   0 & 0 & 0  \\
   d & 0 & -\frac{k}{2}  \\
   k & 0 & m  \\
\end{matrix} \right) (k \neq 0), \quad  
\left( \begin{matrix}
   0 & b & 0  \\
   0 & 0 & \frac{-k}{2}  \\
   k & 0 & m  \\
\end{matrix} \right) (b \cdot k \neq 0).
$$

\qquad {\bf Case 1.3.} Let $l=-2c \neq 0.$ Then we obtain $h=\frac{b d}{4c},$ $k=-\frac{b d}{2c},$ and get the solution 
$$
\left( \begin{matrix}
   0 & b & c  \\
   d & 0 & \frac{bd}{4c}  \\
   -\frac{bd}{2c} & -2c & m  \\
\end{matrix} \right) (c \neq 0).
$$

\textbf{Case 2.} Let $a=-g \neq 0.$ Then $c=h=0$ and we have $b=-\frac{l^2}{4a},$ $d= \frac{4a^3}{l^2},$ $m=0,$ $k = -\frac{4a^2}{l}.$ Thus, we obtain 
$$
\left( \begin{matrix}
   a & -\frac{l^2}{4a} & 0  \\
   \frac{4a^3}{l^2} & -a & 0  \\
   -\frac{4a^2}{l} & l & 0  \\
\end{matrix} \right)(l \cdot g \neq 0).
$$

\textbf{Case 3.} Let $a \neq -g.$ Then we obtain  
Substituting $k=-\frac{4ah}{a+g},$ $l=-\frac{4cg}{a+g},$ $m=\frac{b d-a g-4 c h}{a+g}$ and the following system:
\begin{equation}\label{sys2}
    \left\{\begin{array}{rl}
   \frac{(a-g) \left((a+g)(bd+a^2)-4 a c h\right)}{(a+g)^2}  = 0, \\
   \frac{(a-g) \left(b(a+g)^2 -4 c^2g\right)}{(a+g)^2} = 0, \\
   \frac{(a-g) \left((a^2+bd)c-4c^2h+bh(a+g)\right)}{(a+g)^2} =0, \\
   \frac{(a-g) \left( d(a+g)^2-4 a h^2\right)}{(a+g)^2} =0, \\
   \frac{(a-g) \left((a+g)(bd+g^2) -4 c g h\right)}{(a+g)^2} =0, \\ 
   \frac{(a-g) \left((g^2+bd)h-4ch^2+cd(a+g)\right)}{(a+g)^2}=0.
    \end{array} \right.
\end{equation}

Consider the following two subcases.

\qquad \textbf{Case 3.1.} Let $a=g.$ Then we get the system \eqref{sys2} is hold for any $b, d, c, h$ and $k=-2h,$ $l=-2c,$ $m=\frac{b d-a^2-4 c h}{2a}.$ Thus, we obtain
$$
\left( \begin{matrix}
  a & b & c  \\
  d & a & h  \\
  -2h & -2c & \frac{b d-a^2-4 c h}{2a} \\
\end{matrix} \right)(a \neq 0).
$$

\qquad \textbf{Case 3.2.} Let $a \neq g.$ Subtracting the fifth equation from the first equation of the system \eqref{sys2}, we obtain $h=\frac{(a+g)^2}{4c}, (c \neq 0).$  Substituting to the other equations, we get. $b=\frac{4 c^2 g}{(a+g)^2}, d=\frac{4 a h^2}{(a+g)^2}.$ Hence, we have 
$$
\left( \begin{matrix}
   a & \frac{4 c^2 g}{(a+g)^2} & c  \\
   \frac{a (a+g)^2}{4 c^2} & g & \frac{(a+g)^2}{4 c} \\
  -\frac{a (a+g)}{c} & -\frac{4 c g}{a+g} & -a-g \\
\end{matrix} \right)(c \neq 0, a \neq \pm g).
$$

\end{proof}

\begin{rem}
The following anti-Rota-Baxter operators on $sl_2(\mathbb{C})$ are invertible:
$$\left( \begin{matrix}
   0 & b & 0  \\
   0 & 0 & -\frac{k}{2}  \\
   k & 0 & m  \\
\end{matrix} \right), b \cdot k \neq 0,$$
$$
\left( \begin{matrix}
   0 & b & c  \\
   d & 0 & \frac{bd}{4c}  \\
   -\frac{bd}{2c} & -2c & m  \\
\end{matrix} \right) , \ d (b^3d + 8c^2bm+16c^4) \neq 0 ,
$$ 
$$\left( \begin{matrix}
    g  &  b  & c  \\
    d  &  g  & h  \\
   -2h & -2c & \frac{bd-g^2-4ch}{2g}  \\
\end{matrix} \right), \ b^2 d^2 + 8 c^2 d g - 2 b d g^2 + g^4 - 4 b c d h - 12 c g^2 h + 8 b g h^2 \neq 0.
$$
\end{rem}

\begin{rem}
Any anti-derivation on $sl_2(\mathbb{C})$ in the given basis has the following matrix form \cite{hop}:
$$A=\left( \begin{matrix}
    a_{11}  &  a_{12}  & a_{13}  \\
    a_{21}  &  a_{11}  & a_{23}  \\
    -2a_{23}  &  -2a_{13}  & -2a_{11}  \\
\end{matrix} \right) ,
$$
and inverse of it:
\small
$$\frac{1}{det(A)}\left( \begin{matrix}
    a'  &  b' & c'  \\
    d'  & a'  & h'  \\
   -2h' & -2 c' & \tfrac{b' d'-a'^2-4 c' h'}{2 a'}
\end{matrix} \right), 
$$
where,  $det(A)=-2 a_{11}^3+2 a_{11} a_{12} a_{21}-2 a_{13}^2 a_{21}+4 a_{11} a_{13} a_{23}-2 a_{12} a_{23}^2$ and 

$\begin{array}{lll} a':=-2 (a_{11}^2- a_{13} a_{23}), & b':=2 (a_{11} a_{12}- a_{13}^2), & c':=a_{12} a_{23}-a_{11} a_{13},\\ d':=2 (a_{11} a_{21}- a_{23}^2), & h':=a_{13} a_{21}-a_{11} a_{23}.\end{array}$ 

\end{rem}

% {\small
% \begin{tabular}{p{9cm}}

%    Ibragimov G.I. ,\\
%    Department of Mathematics, Universiti Putra Malaysia, Serdang, Malaysia\\
%     email: ibragimov@upm.edu.my\\

% \end{tabular}
% }

%\label{lastpage}

Name: Azizov Majidkhon \\
Department: Scientific laboratory of algebra and its applications\\
Affiliation: V.I.Romanovskiy Institute  of Mathematics, Uzbekistan Academy of Sciences\\ Tashkent, Uzbekistan.\\
 e-mail: azizovmajidkhan@gmail.com\\    

%\endinput

%
\end{document}